\documentclass[11pt]{amsart}
\usepackage{amssymb,amsthm,amsmath,amstext}
\usepackage{mathrsfs}  
\usepackage{bm}        
\usepackage{mathtools} 

\usepackage[left=1.5in,top=1.2in,right=1.5in,bottom=1.2in]{geometry}

\usepackage{graphicx}

\usepackage[all]{xy}

\theoremstyle{plain}
\newtheorem{theorem}{Theorem}[section]
\newtheorem*{expectation}{Expectation}
\newtheorem{prop}[theorem]{Proposition}
\newtheorem{lemma}[theorem]{Lemma}
\newtheorem{cor}[theorem]{Corollary}

\newtheorem{theoremintro}{Theorem}

\theoremstyle{definition}
\newtheorem{defin}[theorem]{Definition}

\theoremstyle{remark}
\newtheorem{remark}[theorem]{Remark}

\newcommand{\sheaf}[1]{\mathscr{#1}}

\newcommand{\OO}{\sheaf{O}}

\newcommand{\EE}{\sheaf{E}}

\newcommand{\divisor}{\mathrm{div}}

\newcommand{\CliffAlg}{\mathscr{C}}
\newcommand{\CliffZ}{\mathscr{Z}}
\newcommand{\CliffB}{\mathscr{B}}

\newcommand{\isom}{\cong}

\newcommand{\Z}{\mathbb Z}
\newcommand{\A}{\mathbb A}
\newcommand{\C}{\mathbb C}
\renewcommand{\P}{\mathbb P}
\newcommand{\Q}{\mathbb Q}
\newcommand{\Gm}{\mathbb{G}_{\mathrm{m}}}

\DeclareMathOperator{\Aut}{\mathrm{Aut}}
\DeclareMathOperator{\Br}{\mathrm{Br}}

\DeclareMathOperator{\Pic}{\mathrm{Pic}}

\DeclareMathOperator{\SSpec}{\mathbf{Spec}}
\DeclareMathOperator{\Spec}{\mathrm{Spec}}

\DeclareMathOperator{\coker}{\mathrm{coker}}

\newcommand{\inv}{^{-1}}
\newcommand{\mult}{^{\times}}
\newcommand{\dual}{^{\vee}}
\newcommand{\tensor}{\otimes}
\newcommand{\bslash}{\smallsetminus}
\newcommand{\mapto}[1]{\xrightarrow{#1}}

\newcommand{\wt}[1]{\widetilde{#1}}
\newcommand{\et}{\mathrm{\acute{e}t}}

\newcommand{\linedef}[1]{\textsl{#1}}
\newcommand{\Het}{H_{\et}}
\newcommand{\ur}{\mathrm{nr}}
\newcommand{\Hur}{H_{\ur}}
\newcommand{\Pfister}[1]{\ll\!{#1}\gg}
\newcommand{\quadform}[1]{<\! #1 \!>}
\newcommand{\Local}{\mathsf{Local}}
\newcommand{\Ab}{\mathsf{Ab}}

\newcommand{\Frac}{\mathrm{Frac}}
\newcommand{\im}{\mathrm{im}}
\newcommand{\CH}{\mathrm{CH}}

\newcommand{\cores}{\mathrm{cores}}

\newcommand{\Hom}{\mathrm{Hom}}

\usepackage[backref=page]{hyperref}
\hypersetup{pdftitle={Universal unramified cohomology of cubic fourfolds containing a plane}}
\hypersetup{pdfauthor={Asher Auel, Jean-Louis Colliot-Thelene, Parimala}}
\hypersetup{colorlinks=true,linkcolor=blue,anchorcolor=blue,citecolor=blue}

\begin{document}

\title[Universal unramified cohomology of cubic fourfolds]{Universal unramified cohomology of cubic fourfolds containing a plane}

\author[Auel, Colliot-Th\'el\`ene, Parimala]{Asher Auel,
Jean-Louis Colliot-Th\'el\`ene, R.\ Parimala}

\address{Asher Auel, Department of Mathematics, Yale University, 10
Hillhouse Avenue, New Haven, CT 06511, USA \hfill \texttt{\it E-mail
address: \tt asher.auel@yale.edu}}

\address{R. Parimala, Department of Mathematics \& CS, Emory
University, 400 Dowman Drive, Atlanta, GA 30322, USA \hfill \texttt{\it E-mail
address: \tt parimala@mathcs.emory.edu}}

\address{Jean-Louis Colliot-Th\'el\`ene, CNRS, Universit\'e Paris-Sud,
Math\'ematiques, B\^{a}timent 425, 91405 Orsay Cedex, France \hfill \texttt{\it E-mail
address: \tt jlct@math.u-psud.fr}}



\begin{abstract}
We prove the universal triviality of the third unramified cohomology
group of a very general complex cubic fourfold containing a plane.
The proof uses results on the unramified cohomology of quadrics due to
Kahn, Rost, and Sujatha.
\end{abstract}

\maketitle

\section*{Introduction}

Let $X$ be a smooth \linedef{cubic fourfold}, i.e., a smooth cubic
hypersurface in $\P^5$.  A well known problem in algebraic geometry
concerns the rationality of $X$ over $\C$.

\begin{expectation}
The very general cubic fourfold over $\C$ is irrational.
\end{expectation}

Here, ``\linedef{very general}'' is usually taken to mean ``in the
complement of a countable union of proper Zariski closed subsets'' in
the moduli space of cubic fourfolds.  At present, however, not a
single cubic fourfold is provably irrational, though many families of
rational cubic fourfolds are known.

If $X$ contains a plane $P$ (i.e., a linear two dimensional subvariety
of $\P^5$), then $X$ is birational to the total space of a quadric
surface bundle $\wt{X} \to \P^2$ by projecting from $P$.  Its
discriminant divisor $D \subset \P^2$ is a sextic curve.  The
rationality of a cubic fourfold containing a plane over $\C$ is also a
well known problem.

\begin{expectation}
The very general cubic fourfold containing a plane over $\C$ is irrational.
\end{expectation}

Assuming that the discriminant divisor $D$ is smooth, the discriminant
double cover $S \to \P^2$ branched along $D$ is then a K3 surface of
degree 2 and the even Clifford algebra of the quadric fibration
$\wt{X} \to \P^2$ gives rise to a Brauer class $\beta \in \Br(S)$,
called the \linedef{Clifford invariant} of $X$.  This invariant does
not depend on the choice of $P$ (if $\beta$ is not zero, the plane $P$
is actually the unique plane contained in $X$).  By classical results
in the theory of quadratic forms (e.g.,
\cite[Thm.~6.3]{knus_parimala_sridharan:rank_4}), $\beta$ is trivial
if and only if the quadric surface bundle $\wt{X} \to \P^2$ has a
rational section.  Thus if $\beta$ is trivial then $X$ is rational
(see \cite[Thm.~3.1]{hassett:rational_cubic}), though it may happen
that $X$ is rational even when $\beta$ is not trivial (see
\cite[Thm.~11]{ABBV:cubic}).

Some families of rational cubic fourfolds have been described by Fano
\cite{fano}, Tregub \cite{tregub}, \cite{tregub:new}, and
Beauville--Donagi \cite{beauville-donagi}.  In particular,
\linedef{pfaffian cubic fourfolds}, defined by pfaffians of
skew-symmetric $6\times 6$ matrices of linear forms, are rational.
Hassett \cite{hassett:special-cubics} describes, via lattice theory,
Noether--Lefschetz divisors ${\mathcal C}_d$ in the moduli space
${\mathcal C}$ of cubic fourfolds.  In particular, ${\mathcal C}_{14}$
is the closure of the locus of pfaffian cubic fourfolds and ${\mathcal
C}_8$ is the locus of cubic fourfolds containing a plane.  Hassett
\cite{hassett:rational_cubic} identifies countably many divisors of
${\mathcal C}_8$---the union of which is Zariski dense in
$\mathcal{C}_8$---consisting of rational cubic fourfolds with trivial
Clifford invariant.

\medskip

A natural class of birational invariants arise from unramified
cohomology groups.  The unramified cohomology of a rational variety is
trivial (i.e., reduces to the cohomology of the ground field), cf.\
\cite[Thm.~1.5]{colliot:formes_quadratiques_multiplicatives} and
\cite[\S2~and~Thm.~4.1.5]{colliot:santa_barbara}.  Such invariants
have been known to provide useful obstructions to rationality.

For a smooth cubic fourfold $X$ over $\C$, the unramified cohomology
groups
$$
\Hur^i(X/\C, \Q/\Z(i-1))
$$ 
vanish for $0 \leq i \leq 3$, see Theorem~\ref{thm:easy}.  The case
$i=1$ follows from the Kummer sequence (see Theorem~\ref{thm:easy}).
For $i = 2$, one appeals to the Leray spectral sequence and a version
of the Lefschetz hyperplane theorem due to M.\ Noether, as in
\cite[Thm.~A.1]{poonen_voloch}.  For $i=3$, the proof relies on the
integral Hodge conjecture for cycles of codimension 2 on smooth cubic
fourfolds, a result proved by Voisin~\cite[Thm.~18]{voisin:aspects}
building on \cite{murre} and \cite{zucker}.

Something stronger is known when $i \leq 2$, namely that for any field
extension $F/\C$, the natural map $H^i(F,\Q/\Z(i-1)) \to
\Hur^i(X_{F}/F,\Q/\Z(i-1))$ is an isomorphism, see
Theorem~\ref{thm:easy}.  In this case, we say that the unramified
cohomology is \linedef{universally trivial}.  The universal behavior
of an unramified cohomology group can lead to a finer obstruction to
rationality.  While we do not know if the unramified cohomology of an
arbitrary smooth cubic fourfold over $\C$ is universally trivial in
degree 3, our main result is the following.

\begin{theoremintro}
\label{thm:main}
Let $X \subset \P^5$ be a very general cubic fourfold containing a
plane over $\C$.  Then $\Hur^3(X/\C,\Q/\Z(2))$ is universally trivial,
i.e., the natural map $H^3(F,\Q/\Z(2)) \to \Hur^3(X_{F}/F,\Q/\Z(2))$
is an isomorphism for every field extension $F/\C$.
\end{theoremintro}

Here, ``\linedef{very general}'' means that the quadric surface bundle
$\wt{X} \to \P^2$ attached to the plane $P \subset X$ has simple
degeneration (see Proposition~\ref{prop:simpledeg}) and that the
associated K3 surface $r : S \to \P^2$ of degree 2 satisfies $r^* :
\Pic(\P^2) \to \Pic(S)$ is an isomorphism.  A very general cubic
fourfold containing a plane has nontrivial Clifford invariant.
Theorem~\ref{thm:main} implies the corresponding statement with
$\mu_2^{\tensor 2}$ coefficients as well (see
Corollary~\ref{cor:mainZ/2}).

This article is organized as follows.  In \S
\ref{sec:unramified_cohomology}, we expand on the notion of universal
triviality and, more generally, universal torsion by a positive
integer, of the Chow group $A_0(X)$ of zero-cycles of degree zero on a
smooth proper variety $X$.  We state an extension of a theorem of
Merkurjev~\cite[Thm.~2.11]{merkurjev:unramified_cycle_modules} on the
relationship between universal torsion properties of $A_0(X)$ and
analogous properties for unramified cohomology groups and, more
generally, unramified classes in cycle modules.  We also recall that
rationally connected varieties, and among them Fano varieties, satisfy
such a universal torsion property for $A_0(X)$.
 
In \S \ref{sec:cubichyper}, we discuss the specific case of cubic
hypersurfaces $X \subset \P^{n+1}$ for $n \geq 2$.  Already in this
case, the universal triviality is an open question.  We register a
folklore proof that $A_{0}(X)$ is killed by 18, and by 2 as soon as
$X$ contains a $k$-line. In particular, for any smooth cubic
hypersurface $X \subset \P^{n+1}$ over $\C$, with $n \geq 2$, and any
field extension $F/\C$, the group $A_0(X_F)$ is 2-torsion, which also
follows from the existence of a unirational parameterization of $X$ of
degree 2.  The only possible interesting unramified cohomology groups
are thus those with coefficients $\Z/2\Z$. We recall the known results
on these groups.  Our main result, Theorem~\ref{thm:main}, then
appears as one step beyond what was known, but still leaving open the
question of the universal triviality of $A_{0}(X)$.

In \S \ref{sec:Unramified_cohomology_of_quadrics}, we recall some
results on the unramified cohomology groups of quadrics in degrees at
most 3.  The most important is due to Kahn, Rost, and Sujatha (see
Theorem~\ref{thm:KRS}), who build upon earlier work of Merkurjev,
Suslin, and Rost.

In \S \ref{sec:Clifford_algebras}, we discuss the fibration into
2-dimensional quadrics over the projective plane, as well as its
corresponding even Clifford algebra, associated to a smooth cubic
fourfold containing a given plane.

In \S \ref{sec:proof}, we use the results of the previous two sections
to prove Theorem \ref{thm:main}.  The ``very general'' hypothesis
allows us to construct well-behaved parameters at the local rings of
curves on $\P^2$ that split in $S$ (see Lemma~\ref{lem:bad6}).

We thank the American Institute of Mathematics for sponsoring the
workshop ``Brauer groups and obstruction problems:\ moduli spaces and
arithmetic'' held February 25 to March 1, 2013, in Palo Alto,
California.  This work emerges from a problem session group formed
during the workshop, and involved the participation of Marcello
Bernardara, Jodi Black, Evangelia Gazaki, Andrew Kresch, Eric Riedl,
Olivier Wittenberg, and Matthew Woolf. We thank Andrew Kresch for
further discussions after the AIM meeting.

The first author is partially supported by National Science Foundation
grant MSPRF DMS-0903039 and by an NSA Young Investigator grant. The
third author is partially supported by National Science Foundation
grant DMS-1001872.

\section{Unramified cohomology and Chow group of 0-cycles}
\label{sec:unramified_cohomology}

\subsection{Unramified elements}
\label{subsec:unramified}
A general framework for the notion of ``unramified element'' is
established in \cite[\S2]{colliot:santa_barbara}.  Let $k$ be a field
and denote by $\Local_k$ the category of
local $k$-algebras
together with local $k$-algebra homomorphisms.  Let $\Ab$ be the category
of abelian groups and let $M : \Local_k \to \Ab$ be a functor.  For
any field $K/k$ the group of \linedef{unramified elements} of $M$ in
$K/k$ is the intersection
$$ 
M_{\ur}(K/k) = \bigcap_{k \subset \OO \subset K} \im \bigl( M(\OO) \to
M(K) \bigr)
$$
over all rank 1 discrete valuations rings $k \subset \OO \subset K$
with $\Frac(\OO)=K$.

There is a natural map $M(k) \to M_{\ur}(K/k)$ and we say that the
group of unramified elements $M_{\ur}(K/k)$ is \linedef{trivial} if
this map is surjective.

For $X$  an integral scheme of finite type over a field $k$, 
in this paper 
we write $M_{\ur}(X/k) : = M_{\ur}(k(X)/k)$.  
By this definition, the group $M_\ur(X/k)$ is a
$k$-birational invariant of integral schemes of finite type over $k$.
We will be mostly concerned with the functor $M=\Het^i(-,\mu)$ with
coefficients $\mu$ either $\mu_2^{\tensor (i-1)}$ (under the assumption
${\rm char}(k) \neq 2$)
or $$\Q/\Z(i-1) : =
\varinjlim \mu_m^{\tensor (i-1)},$$ the limit being taken over all
integers $m$ coprime to the characteristic of $k$.  In this case,
$M_{\ur}(X/k)$ is called the \linedef{unramified cohomology} group
$\Hur^i(X,\mu)$ of $X$ with coefficients in $\mu$.

\begin{theorem}\label{MSBK}
Let $k$ be a field of  characteristic  different from a prime $p$.
The natural map
$$
H^i(F,\mu_p^{\tensor (i-1)}) \to H^i(F,\Q/\Z(i-1))
$$
is injective.
\end{theorem}

This is a well known consequence of the norm residue isomorphism
theorem (previously known as the Bloch--Kato conjecture) in degree
$i-1$. We use this only for $i=3$, in which case it is a consequence
of the Merkurjev--Suslin theorem.  For $i=3, p=2$, this is a theorem
of Merkurjev.

If $F/k$ is a field extension, we write $X_F = X \times_k F$.  If $X$
is geometrically integral over $k$, we say that $M_{\ur}(X/k)$ is
\linedef{universally trivial} if
$M_{\ur}(X_{F}/F)$ is trivial for every
field extension $F/k$.

\begin{prop}[{\cite[\S2~and~Thm.~4.1.5]{colliot:santa_barbara}}]
\label{prop:Pn}
Let $M : \Local_k \to \Ab$ be a functor satisfying the following
conditions:
\begin{itemize}
\item If $\OO$ is a discrete valuation ring containing $k$, with
fraction field $K$ and residue field $\kappa$, then $\ker\bigl(M(\OO)
\to M(K)\bigr) \subset \ker\bigl(M(\OO) \to M(\kappa)\bigr)$.

\item If $A$ is a regular local ring of dimension 2 containing $k$,
with fraction field $K$, then $\im\bigl(M(A) \to M(K)\bigr) = 
\bigcap_{\mathrm{ht}(\mathfrak{p})=1} \im\bigl(M(A_{\mathfrak{p}}) \to M(K)\bigr)$.

\item The group $M_\ur(\A^1_{k}/k)$ is universally trivial.
\end{itemize}
Then $M_\ur(\P^n_{k}/k)$ is universally trivial. 
\end{prop}

The functor $\Het^i(-,\mu)$ satisfies the conditions of
Proposition~\ref{prop:Pn} (cf.\
\cite[Thm.~4.1.5]{colliot:santa_barbara}), hence if $X$ is a
$k$-rational variety, then $\Hur^i(X/k,\mu)$ is universally trivial.

Denote by $\Ab^{\bullet}$ the category of graded abelian groups.  An
important class of functors $M : \Local_k \to \Ab^{\bullet}$ arise
from the \linedef{cycle-modules} of
Rost~\cite[Rem.~5.2]{rost:cycle_modules}.  In particular, unramified
cohomology arises from a cycle-module.  A cycle module $M$ comes
equipped with residue maps of graded degree $-1$
$$
M^i(k(X)) \mapto{\partial} \bigoplus_{x \in X^{(1)}} M^{i-1}(k(x)) 
$$
for any integral $k$-variety $X$.
If $X$ is smooth and proper, then the kernel
  is    $M^i_{\ur}(X/k)$.  

\subsection{Chow groups of 0-cycles}
Denote by $\CH_d(X)$ the Chow group of $d$-cycles on a smooth variety
$X$ over a field $k$ up to rational equivalence.  If $X$ is proper
over $k$, then there is a well-defined degree map $\CH_0(X) \to \Z$,
and we denote by $A_0(X)$ its kernel, called the Chow group of
0-cycles of degree 0.  The group $A_0(X)$ is a $k$-birational
invariant of smooth, proper, integral varieties over a field $k$, see
\cite[Prop.~6.3]{colliot_coray} requiring resolution of singularities
and \cite[Ex.~16.1.11]{fulton:intersection_theory} in general.  The
computation of the Chow groups of projective space goes back to
Severi.  For 0-cycles, one easily sees $A_0(\P^n_k) = 0$.

For $X$ a proper $k$-variety, we say that $A_0(X)$ is
\linedef{universally trivial} if $A_0(X_F)=0$ for every field
extension $F/k$.  To check triviality of $A_0(X_F)$ over every field
extension $F/k$ seems like quite a burden.  However, usually it
suffices to check it over the function field by the following well
known lemma.

\begin{lemma}
\label{lem:univ_triv}
Let $X$ be a geometrically irreducible smooth proper variety over a
field $k$.  Assume that $X$ has a 0-cycle of degree 1.  The group
$A_0(X)$ is universally trivial if and only if $A_0(X_{k(X)})=0$.
\end{lemma}
\begin{proof}
If $A_0(X)$ is universally trivial then $A_0(X_{k(X)})=0$ by
definition.  Let us prove the converse.  Write $d=\dim(X)$.  Let $\xi
\in X_{k(X)}$ be the $k(X)$-rational point which is the image of the
``diagonal morphism'' $\Spec k(X) \to X \times_k \Spec k(X)$.  Let $P$
be a fixed 0-cycle of degree 1 on $X$.  By hypothesis, we have $\xi =
P_{k(X)}$ in $\CH_0(X_{k(X)})$.  The closures of $P_{k(X)}$ and $\xi$
in $X \times_k X$ are $P \times_k X$ and the diagonal $\Delta_X$,
respectively.  By the closure in $X\times_{k}X$ of a 0-cycle on
$X_{k(X)}$, we mean the sum, taken with multiplicity, of the closures
of each closed point in the support of the 0-cycle on $X_{k(X)}$.
Hence the class of $\Delta_X - P \times_k X$ is in the kernel of the
contravariant map $\CH^d(X \times_k X) \to \CH^d(X_{k(X)})$.  Since
$\CH^d(X_{k(X)})$ is the inductive limit of $\CH^d(X \times_k U)$ over
all dense open subvarieties $U$ of $X$, we have that $\Delta_X - P
\times_k X$ vanishes in some $\CH^d(X \times_k U)$.  We thus have a
decomposition of the diagonal
\begin{equation}
\label{eq:decomp_diagonal}
\Delta_X = P \times_k X + Z
\end{equation}
in $\CH^d(X \times_k X)$, where $Z$ is a cycle with support in $X
\times V$ for some proper closed subvariety $V \subset X$.  

Now, each $d$-cycle $T$ on $X \times_k X$ induces a homomorphism $T_*
: \CH_0(X) \to \CH_0(X)$ defined by $T_*(z) = (p_1)_*(T. p_2^*z)$,
where $p_i : X \times_k X \to X$ are the two projections.  The map $T
\mapsto T_*$ is itself a homomorphism
$$
\CH^d(X \times_k X) = \CH_d(X \times_k X) \to \Hom_\Z(\CH_0(X),\CH_0(X))
$$
by \cite[Cor.~16.1.2]{fulton:intersection_theory}.  We note that
$(\Delta_X)_*$ is the identity map and $(P \times_k X)_*(z) = \deg(z)
P$.  By the easy  moving lemma for 0-cycles on a smooth variety (cf.\ \cite[p.~599]{colliot:finitude}),
 for a proper closed subvariety
$V \subset X$, every 0-cycle on $X$ is rationally equivalent to one
with support away from $V$.  This implies that $Z_* = 0$ for any
$d$-cycle with support on $X \times_k V$ for a proper closed
subvariety $V \subset X$.  Thus if $A_0(X_{k(X)})=0$, then by the
decomposition of the diagonal \eqref{eq:decomp_diagonal}, we see that
the identity map restricted to $A_0(X)$ is zero.

For any field extension $F/k$, we have the base-change $\Delta_{X_F} =
P_F \times_F X_F + Z_F$ of the decomposition of the diagonal
\eqref{eq:decomp_diagonal}, hence the same argument as above shows
that $A_0(X_F)=0$.  We conclude that $A_0(X)$ is universally trivial.
\end{proof}

Let $M$ be a cycle module and let $X$ be a smooth, proper,
geometrically connected variety over the field $k$.  Let $N$ be a
positive integer.  We say that $M_\ur(X/k)$ is \linedef{universally
$N$-torsion} if the cokernel of the natural map $M(F) \to
M_\ur(X_{F}/F)$ is killed by $N$ for every field extension $F/k$. We
say that $A_0(X)$ is \linedef{universally $N$-torsion} if $A_0(X_F)$
is killed by $N$ for every field extension $F/k$.

The \linedef{index} $i(X)$ of a variety $X$ is the
smallest positive degree of a 0-cycle.
 
\begin{theorem}
\label{thm:Merk_N}
Let $X$ be a smooth proper geometrically connected variety over a
field $k$.  Let $N>0$ be an integer.  If the Chow group $A_0(X)$ of
0-cycles of degree 0 is universally $N$-torsion then for every cycle
module $M$ over $k$, the group $M_\ur(X/k)$ is universally $i(X)
N$-torsion.
\end{theorem}
\begin{proof} 
This is a direct consequence of Lemma~\ref{lem:univ_triv} and
\cite[Prop.~RC.8]{KMApp}.
\end{proof}

The case $N=1$ is a generalization of a theorem of
Merkurjev~\cite[Thm.~2.11]{merkurjev:unramified_cycle_modules}, who
also establishes a converse statement. For any $N>0$, one may extend
Merkurjev's method to prove a weak converse:\ if $M_{\ur}(X/k)$ is
universally $N$-torsion for every cycle module $M$ then $A_0(X)$ is
also universally $N$-torsion.

We point out that the appearance of the index of $X$, in the statement
of Theorem~\ref{thm:Merk_N}, is necessary.  For any quadric $X$ over
any field $k$ of characteristic $\neq 2$, we have that $A_0(X)$ is
universally trivial (see \cite{swan:zero-cycles_quadrics}).  However,
if $X$ is the 4-dimensional quadric associated to an anisotropic
Albert form $X$ over a field $k$ of characteristic $\neq 2$, then
$\coker\bigl( H^3(F,\Q/\Z(2)) \to \Hur^3(X/k,\Q/\Z(2)) \bigr)\isom
\Z/2\Z$ by \cite[Thm.~5]{kahn_rost_sujatha}.  Hence $A_0(X)$ is
universally trivial, while there are nontrivial unramified elements of
some cycle module over $X$.  Note that the index of an anisotropic
quadric is 2.

The following lemmas will be used in the next section.

\begin{lemma}
\label{lem:finite_sufficient}
Let $X$ be a proper variety over a field $k$ with $X(k) \neq
\emptyset$.  Let $N>0$ be an integer.  If for every finite
extension $K/k$ and any two $K$-points $P,Q \in X(K)$
the class of the  0-cycle $P-Q$
is $N$-torsion in $A_{0}(X_{K})$,  then $A_0(X)$ is $N$-torsion.
\end{lemma}
\begin{proof}
Fixing $P \in X(k)$, any element of $A_0(X)$ can be written as a
linear combination of 0-cycles of degree 0 of the form $Z - \deg(Z)P$,
where $Z$ is a closed point of $X$.  Let $K$ be the residue field of
$Z$.  Consider the morphism $f : X_K \to X$.  Since $K \tensor_k K$
has $K$ as a direct factor, there is a corresponding $K$-rational
point $\zeta$ of $X_K$ lying over $Z$, such that $f_* \zeta = P$.  By
hypothesis $\zeta - P_K$ is $N$-torsion, hence $f_*(\zeta - P_K) =
Z-\deg(Z)P$ is $N$-torsion.
\end{proof}

\begin{lemma}
\label{lem:inject}
Let $k$ be an algebraically closed field and $K/k$ a field extension.
Let $X$ be a smooth projective connected variety over $k$.
Then the natural map $\CH_0(X) \to \CH_0(X_K)$ is injective.
\end{lemma}
\begin{proof}
Let $z$ be a 0-cycle on $X$ that becomes rationally equivalent to zero
on $X_K$.  Then there exists a subextension $L$ of $K/k$ that is
finitely generated over $k$, such that $z$ becomes rationally
equivalent to zero on $X_L$. In fact, we can spread out the rational
equivalence to a finitely generated $k$-algebra $A$.  Since $k$ is
algebraically closed, there are many $k$-points on $\Spec A$, at which
we can specialize the rational equivalence.
\end{proof}

\begin{lemma}\label{infinitetranscendance}
\label{lem:absolute_gives_N}
Let $X$ be a smooth proper connected variety over an algebraically
closed field $k$ of infinite transcendence degree over its prime
field.  If $A_0(X)=0$ then there exists an integer $N>0$ such that
$A_0(X)$ is universally $N$-torsion.
\end{lemma}
\begin{proof}
The variety $X$ is defined over an algebraically closed subfield $L
\subset k$, with $L$ algebraic over a field finitely generated over
its prime field.  That is, there exists a variety $X_0$ over $L$ with
$X \isom X_0 \times_L k$.  Let $\eta$ be the generic point of $X_0$.
Let $P$ be an $L$-point of $X_0$.  One may embed the function field
$F=L(X_0)$ into $k$, by the transcendence degree hypothesis on $k$.
Let $K$ be the algebraic closure of $F$ inside $k$.  By
Lemma~\ref{lem:inject} and the hypothesis that $A_0(X)=0$, we have
that $A_0(X_0 \times_L F)=0$.  This implies that there is a finite
extension $E/F$ of fields such that $\eta_E - P_E =0$ in $A_0(X_0
\tensor_L E)$.  Taking the corestriction to $F$, one finds that
$N(\eta_F-P_F)=0$ in $A_0(X_0\times_L F)$, hence in $A_0(X)$ as well.
Arguing as in the proof of Lemma~\ref{lem:univ_triv}, we conclude that
$A_0(X)$ is universally $N$-torsion.
\end{proof}

\bigskip

\subsection{Connections with complex geometry}

The universal torsion of $A_0(X)$ puts strong restrictions on the
variety $X$.  For example, the following result is well known.

\begin{prop}
\label{prop:bloch}
Let $X$ be a smooth proper geometrically irreducible variety over a
field $k$ of characteristic zero.  If $A_0(X)$ is universally
$N$-torsion for some positive integer $N$ then,
$H^0(X,\Omega^{i}_{X})=0$ and $H^i(X,\OO_X)=0$ for all $i \geq 1$.
\end{prop}
\begin{proof}
Over a complex surface, the result goes back to
Bloch~\cite[App.~Lec.~1]{bloch:lectures}, by exploiting a
decomposition of the diagonal and the action of cycles on various
cohomology theories.  Aspects of this argument were developed in
\cite{bloch_srinivas}.  A proof over the complex numbers can be found
in \cite[Cor.~10.18,~\S10.2.2]{voisin:Hodge}.  Over a general field of
characteristic zero, the argument is sketched in
\cite[p.~187]{esnault:finite_field_trivial_Chow}.
\end{proof}

Over $\C$, the universal triviality of $A_0(X)$ does not imply $H^0(X,
\omega_X^{\tensor n})=0$ for all $n > 1$.  Otherwise, a surface over
$\C$ with $A_0(X)=0$ would additionally satisfy $P_2(X) =
h^0(X,\omega_X^{\tensor 2})=0$, hence would be rational by
Castelnovo's criterion.

It is however well known that there exist nonrational complex surfaces
$X$, with $p_g(X)=q(X)=0$ and for which $A_0(X)=0$.  Enriques surfaces
were the first examples, extensively studied in \cite{enriques},
\cite[p.~294]{enriques:memorie} with some examples considered earlier
in \cite{reye}, see also \cite{castelnuovo}.  For more examples, see
\cite{BKL}.  We remark that for an Enriques surface $X$, we have that
$\Hur^1(X/\C,\Z/2\Z)=H^1_{\et}(X,\Z/2\Z)=\Z/2\Z$.  Hence by
Theorem~\ref{thm:Merk_N}, we see that $A_{0}(X)$ is not universally
trivial.

\medskip

The following result was stated without detailed proof as the last
remark of \cite{bloch_srinivas}.  As we show, it is an immediate
consequence of a result in
\cite{colliot-thelene_raskind:second_Chow_group}.  A more geometric
proof was recently shown to us by C.~Voisin.

\begin{prop}
\label{surfacenotorsion}
Let $X$ be a smooth proper connected surface over $\C$.  Suppose that
all groups $H^i_{\mathrm{Betti}}(X(\C),\Z)$ are torsionfree and that
$A_0(X)=0$.  Then $A_0(X)$ is universally trivial.
\end{prop}
\begin{proof}
By Lemma~\ref{lem:absolute_gives_N}, we have that $A_0(X)$ is
universally $N$-torsion.  Hence by Lemma~\ref{prop:bloch}, we have
that $H^i(X,\OO_X)=0$ for all $i \geq 1$.  Hence $p_g(X) = q(X) = 0$,
and thus $b_3(X) = b_1(X) = 2q(X)=0$.  The torsion-free hypothesis on
cohomology finally allows us conclude, from
\cite[Thm.~3.10(d)]{colliot-thelene_raskind:second_Chow_group}, that
$A_0(X_F)=0$ for any field extension $F/\C$.
\end{proof}

\begin{cor}\label{surfacenotorsionbis}
Let $X$ be a smooth proper connected surface over $\C$.  Suppose that
the N\'eron-Severi group $\mathrm{NS}(X)$ of $X$ is torsionfree and
that $A_{0}(X)=0$. Then $A_0(X)$ is universally trivial.
\end{cor}
\begin{proof} 
The group $H^1_{\mathrm{Betti}}(X(\C),\Z)$ is clearly torsionfree.
For any smooth proper connected variety $X$ over $\C$ of dimension $d$
there is an isomorphism $NS(X)_{tors} \simeq
H^2_{\mathrm{Betti}}(X(\C),\Z)_{tors}$, and the finite groups
$H^2_{\mathrm{Betti}}(X(\C),\Z)_{tors}$ and
$H^{2d-1}_{\mathrm{Betti}}(X(\C),\Z)_{tors}$ are dual to each other.
The hypotheses thus imply that for the surface $X$, all groups
$H^i_{\mathrm{Betti}}(X(\C),\Z)$ are torsionfree, and we apply
Proposition~\ref{surfacenotorsion}.
\end{proof}

The first surfaces $S$ of general type with $p_g(S)=q(S)=0$ were
constructed in \cite{campedelli} and \cite{godeaux}.  Simply connected
surfaces $X$ of general type for which $p_g(X)=0$ were constructed by
Barlow~\cite{barlow}, who also proved that $A_0(X)=0$ for some of
them.  See also the recent paper \cite{voisin:bloch}.  For such
surfaces, $\Pic(X) ={\rm NS}(X)$ has no torsion, thus Corollary
\ref{surfacenotorsionbis} applies.  The group $A_{0}(X)$ is
universally trivial, but the surfaces are far from being rational,
since they are of general type.

\bigskip

A smooth projective variety $X$ over a field $k$ is called
\linedef{rationally chain connected} if for every algebraically closed
field extension $K/k$, any two $K$-points of $X$ can be connected by a
chain of rational curves.  Smooth, geometrically unirational varieties
are rationally connected.  It is a theorem of
Campana~\cite{campana:Fano} and
Koll\'ar--Miyaoka--Mori~\cite{kollar_miyaoka_mori:Fano} that any
smooth projective Fano variety is rationally chain connected.  If $X$
is rationally chain connected, then $A_0(X_K)=0$ for any algebraically
closed field extension $K/k$.  While a standard argument then proves
that $A_0(X_F)$ is torsion for every field extension $F/k$, the
following more precise result, in the spirit of Lemma
\ref{infinitetranscendance} above, is known.

\begin{prop}[{\cite[Prop.~11]{colliot:finitude}}]
\label{prop:rat_conn_N}
Let $X$ be a smooth, projective, rationally chain connected variety
over a field $k$.  Then there exists an integer $N>0$ such that $A_0(X)$
is universally $N$-torsion.
\end{prop}

There exist rationally connected varieties $X$ over an algebraically
closed field of characteristic zero with $A_0(X)$ not universally
trivial.  Indeed, let $X$ be a unirational threefold with
$\Hur^2(X,\Q/\Z(1)) \isom \Br(X) \neq 0$, see e.g.,
\cite{artin_mumford}.  Then by 
Theorem~\ref{thm:Merk_N}, $A_0(X)$ is not universally trivial. 

We do not know whether there exist smooth Fano varieties over an
algebraically closed field of characteristic zero with $A_0(X)$ not
universally trivial.

\section{Chow groups of 0-cycles on cubic hypersurfaces}
\label{sec:cubichyper}

Now we will discuss the situation for   cubic hypersurfaces $X
\subset \P^{n+1}$ with $n \geq 2$. 
Quite a few years ago, one of us learned the following argument from
the Dean of Trinity College, Cambridge.

\begin{prop}
\label{prop:6}
Let $X \subset \P^{n+1}_{k}$, with $n \geq 2$, be an arbitrary  cubic hypersurface
over an arbitrary     field $k$.
 Then $A_0(X)$ is  
18-torsion.  If $X(k) \neq \emptyset $ then $A_0(X)$ is  
6-torsion.  If $X$ contains a $k$-line, then $A_0(X)$ is  
2-torsion.
\end{prop}
\begin{proof}
Assuming the assertion in the case $X(k) \neq \emptyset$, we can
deduce the general case from a norm argument, noting that $X$ acquires
a rational point after an extension of degree 3.  So we assume that
$X(k) \neq \emptyset$.  By Lemma~\ref{lem:finite_sufficient}, to prove
the general statement over an arbitrary field $k$, it suffices to
prove that for any distinct points $P,Q \in X(k)$, the class of the
0-cycle $P-Q$ is 6-torsion, and that it is 2-torsion if $X$ contains a
$k$-line.  Choose a $\P^3_{k} \subset \P^{n+1}_{k}$ containing $P$ and
$Q$.  If $X$ contains a $k$-line $L$, take such a $\P^3_{k}$ which
contains the line $L$.  Then $X \cap \P^3_{k}$ is either $\P^3_{k}$ or
is a cubic surface in $\P^3_{k}$.  In the first case, $P$ is
rationally equivalent to $Q$ on $\P^3_{k}$, hence on $X$.  Thus we can
assume, for the rest of the proof, that $X \subset \P^3_{k}$ is a
cubic surface, possibly singular.  We will sketch a route to prove the
result, leaving the minute details to the reader.

We first show that if the cubic surface $X$ contains a $k$-line $L$,
then for any two distinct points $P,Q \in X(k)$, we have $2(P-Q)=0 \in
A_{0}(X)$. It is enough to prove this when $Q$ is a $k$-point on $L$.
If $P$ also lies on $L$, then $P-Q=0$.  If $P$ does not lie on $L$,
let $\Pi$ be the plane spanned by $P$ and $L$.  If it is contained in
$X$, then clearly $P-Q=0$. If not, then it cuts out on $X$ a cubic
curve $C$, one component of which is the line $L$, the other component
is a conic. At this juncture, we leave it to the reader to consider
the various possible cases and show that $2(P-Q)$ is rationally
equivalent to zero on $C$, hence on $X$.  The coefficient $2$ is the
degree of intersection of the conic with the line.

We may now assume that the cubic surface $X$ does not contain a
$k$-line.  If $P$ and $Q$ were both singular, then the line through
$P$ and $Q$ would intersect the surface with multiplicity at least 4,
hence would be contained in $X$, which is excluded.  We may thus
assume that $P$ is a regular $k$-point.  First assume that $Q$ is
singular.  Let $\Pi$ be a plane which contains $P$ and $Q$. It is not
contained in $X$.  Its trace on $X$ is a cubic curve $C$ in the plane
$\Pi$, such that $C$ is singular at $Q$.  A discussion of cases then
shows that $6(Q-P)=0$ on $C$, hence on $X$ (the occurence of $6$
rather than $2$, comes from allowing nonperfect fields in
characteristic 3).  Now suppose that both $P$ and $Q$ are regular,
hence smooth, $k$-points.  Let $T_{P} \subset \P^3_{k}$ be the tangent
plane to $X$ at $P$, and $T_{Q} \subset \P^3_{k}$ the tangent plane to
$X$ at $Q$.  Since $X$ contains no $k$-line, the tangent planes are
distinct. Let $L=T_{P}\cap T_{Q}$ be their intersection. This line is
not contained in $X$, which it intersects in a zero-cycle $z$ of
degree 3 over $k$.  The trace of $X$ on the plane $T_{P}$ is a cubic
curve $C_{P}$ which is singular at $P$ and contains the 0-cycle $z$.
We leave it to the reader to check that $6P-2z=0 \in A_{0}(C_{P})$,
hence in $A_{0}(X)$.  Similarly, $6Q-2z=0 \in A_{0}(C_{Q})$, hence in
$A_{0}(X)$. Thus $6(P-Q)=0 \in A_{0}(X)$ in all cases.
\end{proof}

For smooth cubic hypersurfaces, the last part of
Proposition~\ref{prop:6} is a consequence of the fact that a cubic
hypersurface containing a line has unirational parameterizations of
degree 2.  This fact that was likely known to M.\ Noether (cf.\
\cite[App.~B]{clemens_griffiths}).

\begin{theorem}
\label{thm:unirat}
Let $X \subset \P^{n+1}_{k}$, with $n \geq 2$, be a smooth cubic
hypersurface containing a $k$-line $L$.  Then $X$ has a unirational
parameterization of degree 2.
\end{theorem}
\begin{proof}
Denote by $W$ the variety of pairs $(p,l)$ where $p \in L$ and $l$ is
a line in $\P^{n+1}_{k}$ tangent to $X$ at $p$.  Then the projection $W
\to L$ is a Zariski locally trivial $\P^{n-1}$-bundle.  A general such
line $l$ intersects $X$ in one further point, defining a rational map
$g : W \to X$.  This map is two-to-one.  Indeed, as before, for a
general point $p \in X$, the plane through $p$ and $L$ meets $X$ in
the union of $L$ and a smooth conic. Generally, that conic meets $L$
in two points.  The lines through $p$, and tangent to $X$, are exactly
those connecting $p$ to these two points of intersection.

We can give another argument, following
\cite[\S2.1]{harris_mazur_pandharipande}.  Projecting from $L$
displays $X$ as birational to the total space of a conic bundle $Y \to
\P^{n-1}_{k}$, where $Y$ is the blow-up of $X$ in $L$.  Each point $P$ in
the base $\P^{n-1}_{k}$ corresponds to a plane containing $L$, and
intersecting $X$ with this plane is the union of $L$ and a conic
$C_P$; this conic is the fiber above $P$ in the conic bundle.  Let $M$
be the incidence variety of pairs $(P,p)$ where $P \in \P^{n-1}_{k}$ and
$p \in L$, such that $p \in C_P$.  Then the projection $M \to
\P^{n-1}_{k}$ has degree 2, the fiber over each $P$ is exactly the two
points of intersection of $L$ and $C_P$.  The other projection $M \to
L$ displays $M$ as the total space of a Zariski locally trivial
$\P^{n-2}$-bundle, hence $M$ is rational.  Then consider the base
change of $Y \to \P^{n-1}_{k}$ by $M \to \P^{n-1}_{k}$.  The resulting conic
bundle $M \times_{\P^{n-1}_{k}} Y \to M$ has a tautological rational
section.  Thus $M \times_{\P^{n-1}_{k}} Y$ is rational (being a conic
bundle over $M$ with a rational section) and the projection to $Y$ has
degree 2.
\end{proof}

Over an algebraically closed field, any cubic hypersurface $X \subset
\P^{n+1}$, with $n \geq 2$, contains a line.  Indeed, by taking
hyperplane sections, one is reduced to the well knwon fact that any
cubic surface $X \subset \P^3$ contains a line over an algebraically
closed field, see
\cite[Prop.~7.2]{reid:undergraduate_algebraic_geometry} for instance.
Proposition \ref{prop:6} and Theorem \ref{thm:Merk_N} then yield the
following corollary.

\begin{cor}
\label{cor:univ2}
Let $X \subset \P^{n+1}_{k}$, with $n \geq 2$, be a smooth cubic
hypersurface over a field $k$.  If $X$ contains a $k$-line (e.g., if
$k$ is an algebraically closed field) then $A_0(X)$ is universally
2-torsion and thus for every cycle module $M$ over $k$, the group
$M_\ur(X/k)$ is universally 2-torsion.
\end{cor}

Hassett and
Tschinkel~\cite[\S7.5]{hassett_tschinkel:rational_curves_homolomorphic_symplectic}
prove that the cubic fourfolds over $\C$ with a unirational
parameterization of odd degree are Zariski dense in the moduli space.
In particular, they prove that on a Zariski dense subset of the
Noether--Lefschetz divisor $\mathcal{C}_{d}$, with $d=2(m^2+m+1)$ for
$m \geq 1$, the cubic fourfolds have a unirational parameterization of
degree $m^2+m+1$.  The countable union of these loci is dense in the
moduli space.  Together with Corollary~\ref{cor:univ2}, this implies
that such cubic fourfolds $X$ have universally trivial $A_0(X)$, and
thus for every cycle module $M$ over $\C$, the group $M_\ur(X/k)$ is
universally trivial.

Corollary~\ref{cor:univ2} leaves the following questions open. Let $X
\subset \P^{n+1}$, with $n \geq 2$, be a smooth cubic hypersurface
over $\C$.
\begin{enumerate}
\item Is the group $A_{0}(X)$ universally trivial?

\item For any integer $i \geq 1$, are the unramified cohomology groups
$\Hur^{i}(X/\C, \mu_2^{\tensor (i-1)})$ universally trivial?
\end{enumerate}

\medskip

The following theorem gathers previously known results.

\begin{theorem}
\label{thm:easy}
Let $X \subset \P^{n+1}_{k}$ be a smooth cubic hypersurface over a field
$k$ of characteristic zero.  Then $\Hur^i(X/k,\Q/\Z(1))$ is
universally trivial for all $n \geq 3$ and $0 \leq i \leq 2$.  If
$k=\C$, then $\Hur^3(X/k,\Q/\Z(2))$ is trivial for all $3 \leq n \leq
4$.
\end{theorem}
\begin{proof}
Let $F/k$ be any field extension.  For any complete intersection $Y
\subset \P^{n+1}_{F}$ of dimension $\geq 3$ over $F$, the restriction
map on Picard groups $\Pic(\P^{n+1}_{F}) \to \Pic(Y)$ is an
isomorphism and the natural map on Brauer groups $\Br(F) \to \Br(Y)$
is an isomorphism, see \cite[Thm.~A.1]{poonen_voloch}.  By purity, for
any smooth variety $Y$ over $F$, we have that $\Hur^1(Y/F,\Q/\Z(1)) =
\Het^1(Y,\Q/\Z(1))$ and that $\Hur^2(Y/F,\Q/\Z(1)) = \Br(Y)$, see
\cite[Cor.~3.2,~Prop.~4.1]{colliot_sansuc:multiplicatif} and
\cite{bloch_ogus}.  From the Kummer sequence for a projective and
geometrically connected variety $X$ over $F$, we get an the exact
sequence
$$ 
1 \to F\mult/F^{\times n} \to \Het^1(X,\mu_n) \to \Pic(X)[n] \to 0
$$  
which yields an exact sequence
$$
0 \to H^1(F,\Q/\Z(1)) \to \Het^1(X,\Q/\Z(1)) \to
\Pic(X)_{\mathrm{tors}} \to 0.
$$
upon taking direct limits.

When $X \subset \P^{n+1}_{k}$ is a smooth cubic hypersurface, the above
considerations imply that $\Hur^i(X/k,\Q/\Z(1))$ is universally
trivial for $i=1,2$.  Also $\Hur^0(X/k,\Q/\Z(1))$ is universally
trivial since $X$ is geometrically irreducible.

Now assume $k=\C$.  If $n=3$, then $X$ contains a line hence is
birational to a conic fibration over $\P^2$, as in the proof of
Theorem~\ref{thm:unirat}.  For a conic bundle $Y$ over a complex
surface, one has $\Hur^3(Y/\C,\Q/\Z(2))=0$ (cf.\
\cite[Cor.~3.1(a)]{colliot:quadrics}).  If $n=4$ and $X$ contains a
plane, then $X$ is birational to a fibration $Y \to \P^2$ in
2-dimensional quadrics, and once again
\cite[Cor.~3.1(a)]{colliot:quadrics} yields $\Hur^3(X/\C,\Q/\Z(2))=0$.

For $n \geq 2$ arbitrary, $X$ is unirational hence rationally chain
connected.  Then \cite[Thm.~1.1]{colliot_voisin:unramified_cohomology}
implies that the integral Hodge conjecture for codimension 2 cycles on
$X$ is equivalent to the vanishing of $\Hur^3(X/\C,\Q/\Z(2))$.

For smooth cubic threefolds, the integral Hodge conjecture holds for
codimension 2 cycles, as $H^4(X,\Z)$ is generated by a line. This
yields another proof of $\Hur^3(X/\C,\Q/\Z(2))=0$ in the case $n=3$.

For smooth cubic fourfolds, the integral Hodge conjecture for
codimension 2 cycles is a result of
Voisin~\cite[Thm.~18]{voisin:aspects}, building on \cite{murre} and
\cite{zucker}. We thus get $\Hur^3(X/\C,\Q/\Z(2))=0$ for an arbitrary
smooth cubic hypersurface $X \subset \P^5$.
\end{proof}

Our main result, Theorem \ref{thm:main}, is a first step beyond the
mentioned results.


\section{Unramified cohomology of quadrics}
\label{sec:Unramified_cohomology_of_quadrics}

Let $Q$ be a smooth quadric over a field $k$ of characteristic $\neq
2$ defined by the vanishing of a quadratic form $q$.  We note that the
dimension of $Q$ (as a $k$-variety) is 2 less than the dimension of
$q$ (as a quadratic form).  When $Q$ has even dimension, one defines
the \linedef{discriminant} $d(Q) \in H^1(k,\mu_2)$ of $Q$ to be the
(signed) discriminant of $q$.  If $Q$ has even dimension and trivial
discriminant or has odd dimension, then define the \linedef{Clifford
invariant} $c(Q) \in \Br(k)$ of $Q$ to be the Clifford invariant of
$q$, i.e., the Brauer class of the Clifford algebra $C(q)$ or the even
Clifford algebra $C_0(q)$, respectively.  We point out when $q$ has
even rank and trivial discriminant, then a choice of splitting of the
center induces a decomposition $C_0(q) \isom C_0^+(q) \times
C_0^-(q)$, with $C(q)$, $C_0^+(q)$, and $C_0^-(q)$ all Brauer
equivalent central simple $k$-algebras.  Under the given constraints
on dimension and discriminant, these invariants only depend on the
similarity class of $q$, and thus yield well-defined invariants of
$Q$.

The following two results are well known
(cf. \cite[\S5,~p.~485]{arason}; see also the proof of
\cite[Th\'eor\`eme~2.5]{colliot_skorobogatov:quadriques}),
though we could not find the second stated in the literature.

\begin{theorem}
\label{thm:kernel}
Let $k$ be a field of characteristic $\neq 2$.  Let $Q$ be a smooth
quadric surface over $k$.  Then
$$
\ker\bigl( \Br(k) \to \Br(k(Q)) \bigr) = 
\begin{cases}
0 & \text{if $d(Q)$ is nontrivial} \\
\Z/2\Z \cdot c(Q) & \text{if $d(Q)$ is trivial}
\end{cases}
$$
 where $c(Q)$ is the Clifford invariant of $Q$.  
\end{theorem}

\begin{prop}
\label{prop:kernel_cone}
Let $k$ be a field of characteristic $\neq 2$.  Let $Q$ be a quadric
surface cone over $k$, the base of which is a smooth conic $Q_0$.  Then
$$
\ker\bigl( \Br(k) \to \Br(k(Q)) \bigr) = 
\Z/2\Z \cdot c(Q_0)
$$
where $c(Q_0)$ is the Clifford invariant of $Q_0$.  
\end{prop}
\begin{proof}
In this case $k(Q) \isom k(Q_0 \times_{k} \P^1_{k})$, hence $\ker\bigl( \Br(k)
\to \Br(k(Q)) \bigr)$ equals $\ker\bigl( \Br(k) \to \Br(k(Q_0)) \bigr)$.
Thus the proposition follows from the case of smooth conics, a result
going back to Witt~\cite[Satz~p.~465]{witt:norm}.
\end{proof}

Finally, the deepest result we will need is the following one
concerning the degree three unramified cohomology of a quadric.

\begin{theorem}[{Kahn--Rost--Sujatha~\cite[Thm.~5]{kahn_rost_sujatha}}]
\label{thm:KRS}
Let $k$ be a field of characteristic $\neq 2$.  Let $Q$ be a smooth
quadric surface over $k$.
Then the natural map
$$
H^3(k,\Q/\Z(2)) \to \Hur^3(Q/k,\Q/\Z(2))
$$
is surjective.
\end{theorem}

The following is an amplification of one the main results  of Arason's thesis.
\begin{theorem}
\label{thm:arason}
Let $k$ be a field of characteristic $\neq 2$.  Let $Q$ be a smooth
quadric surface over $k$ defined by a nondegenerate quadratic form $q$
of rank 4. Then the kernel of the  map 
$H^3(k,\Q/\Z(2)) \to H^3(k(Q),\Q/\Z(2))$
coincides with the kernel of the map
$H^3(k,\mu_{2}^{\otimes 2}) \to
H^3(k(Q),\mu_{2}^{\otimes 2})$, and it is equal to  the set of symbols
$$
\{ (a,b,c) \; : \; q ~\text{is similar to a subform of}~ \Pfister{-a,-b,-c} \}.
$$
\end{theorem}
\begin{proof}
By a standard norm argument, 
the kernel of $H^3(k,\Q/\Z(2)) \to H^3(k(Q),\Q/\Z(2))$
is 2-torsion.  By Merkurjev's theorem (see Theorem \ref{MSBK}),
the two kernels in the Theorem thus  coincide. The precise description of the kernel 
with coefficients $\Z/2\Z$ is Arason's ~\cite[Satz~5.6]{arason}.
\end{proof}

However, for our purposes, we will only need to know that certain
special symbols are contained in this kernel.  We can give a direct
proof of this fact.

\begin{lemma}
\label{lem:ct}
Let $k$ be a field of characteristic $\neq 2$.  If
$\quadform{1,-a,-b,abd}$ is isotropic over $k$, then for $w$ any norm
from $k(\sqrt{d})/k$, the symbol $(a,b,w) \in H^3(k,\mu_2^{\tensor
2})$ is trivial.
\end{lemma}
\begin{proof}
Put $l = k(\sqrt{d})$.  As $\quadform{1,-a,-b,abd}$ is isotropic,
there exist $x,y,u,v \in k$ such that $x^2-ay^2= b(u^2-adv^2)\neq 0$.
The class $(a,x^2-ay^2)$ is trivial, hence
$$
(a,b,w) = (a,(x^2-ay^2)/(u^2-adv^2),w) = (a,u^2-adv^2, w).
$$
Let $w = N_{l/k}(w')$. By the projection formula, we have that
$$
(a,b,w) = \cores_{k(\sqrt{d})/k}(a,u^2-adv^2,w')
$$
which is trivial since $(a,u^2-adv^2) = (a,u^2-a(\sqrt{d}v)^2) \in H^2(l, \mu_2^{\tensor 2})$ is trivial.
\end{proof}


\section{Cubic fourfolds containing a plane and Clifford algebras}
\label{sec:Clifford_algebras}

Let $X$ be a smooth cubic fourfold over a field $k$.  Suppose $X
\subset \P^5_{k}=\P(V)$ contains a plane $P = \P(W)$, where $W \subset
V$ is a dimension 3 linear subspace of $V$.  Let $\wt{X}$ be the
blow-up of $X$ along $P$ and $\pi : \wt{X} \to \P(V/W)$ the projection
from $P$.  We will write $\P^2_{k} = \P(V/W)$.  Then the blow-up of
$\P^5_{k}$ along $P$ is isomorphic to the total space of the
projective bundle $p : \P(\EE) \to \P^2_{k}$, where $\EE =
(W\tensor\OO_{\P^2_{k}} ) \oplus \OO_{\P^2_{k}}(-1)$, and in which
$\pi : \wt{X} \to \P^2_{k}$ embeds as a quadric surface bundle.

Now choose homogeneous coordinates $(x_0:x_1:x_2:y_0:y_1:y_2)$ on
$\P^{5}_{k}$.  Since $\Aut_k(\P^5)$ acts transitively on the set of planes
in $\P^5_{k}$, without loss of generality, we can assume that $P = \{
x_0=x_1=x_2=0\}$. Write the equation of $X$ as
$$
\sum_{0 \leq m \leq n \leq 2} a_{mn} y_m y_n +
\sum_{0 \leq p \leq 2} b_p y_p + c = 0
$$ 
for homogeneous linear polynomials $a_{mn}$, quadratic polynomials
$b_p$, and a cubic polynomial $c$ in $k[x_0,x_1,x_2]$.  Then we 
define a quadratic form $q : \EE \to \OO_{\P^2_{k}}(1)$ over $\P^2_{k}$ by
\begin{equation}
\label{eq:q}
q(y_0,y_1,y_2,z) = \sum_{0 \leq m \leq n \leq 2} a_{mn} y_m y_n +
\sum_{0 \leq p \leq 2} b_p y_p z + cz^2
\end{equation}
on local sections $y_i$ of $\OO_{\P^2_{k}}$ and $z$ of
$\OO_{\P^2_{k}}(-1)$.  Of course, given $X$, the quadratic form $q$ is
only well-defined up to multiplication by an element of $\Gamma(X,\Gm)
= k\mult$. The quadric fibration associated to
$(\EE,q,\OO_{\P^2_{k}}(1))$ is precisely $\pi : \widetilde{X} \to
\P^2_{k}$.  The associated bilinear form $b_q : \EE \to
\EE\dual\tensor \OO_{\P^2_{k}}(1)$ has Gram matrix
\begin{equation}
\label{eq:b}
\begin{pmatrix}
2a_{00} & a_{01} & a_{02} & b_0 \\
a_{01} & 2a_{11} & a_{12} & b_1 \\
a_{02} & a_{12} & 2a_{22} & b_2 \\
b_0 & b_1 & b_2 & 2c
\end{pmatrix}
\end{equation}
whose determinant $\Delta$ is a homogeneous sextic polynomial defining
the discriminant divisor $D \subset \P^2_{k}$.

\begin{prop}
\label{prop:simpledeg}
Let $X$ be a smooth cubic fourfold containing a plane $P$ over a field
$k$ of characteristic $\neq 2$.  Denote by $\pi : \wt{X} \to \P^2_{k}$ the
associated quadric surface bundle, $D \subset \P^2_{k}$ the discriminant
divisor, and $U = \P^2_{k} \bslash D$.  Then the following are equivalent:
\begin{enumerate}
\item\label{sd.a} The divisor $D$ is smooth 
over $k$.

\item\label{sd.b} The fibers of $q$ are nondegenerate over points of
$U$ and have a radical of dimension 1 over points of $D$.

\item\label{sd.c} The fibers of $\pi$ are smooth quadric surfaces over
points of $U$ and are quadric surface cones with isolated singularity
over points of $D$.

\item\label{sd.d} There is no other plane in $X \times_k \overline{k}$
meeting $P\times_k \overline{k}$.
\end{enumerate}
In this case, we say that $\pi$ has \linedef{simple degeneration}.
\end{prop}
\begin{proof}
The equivalence between \ref{sd.a} and \ref{sd.b} is proved in
\cite[I~Prop.~1.2(iii)]{beauville:prym_varieties_intermediate_jacobian}
over an algebraically closed field and \cite[Prop.~1.6]{ABB:fibrations} in
general.  The equivalence of \ref{sd.b} and \ref{sd.c} follows from
the fact that the singular locus of a quadric is the projectivization
of its radical.  The statement that \ref{sd.d} implies \ref{sd.a}
appears without proof in \cite[\S1~Lemme~2]{voisin}, and holds over a
general field.  Finally, another plane intersecting $P$ nontrivially
will give rise, in the projection, to a singular line or plane in a
fiber of the quadric fibration, contradicting \ref{sd.c}.
\end{proof}

Let $\CliffAlg_0$ be the even Clifford algebra associated to
$(\EE,q,\OO_{\P^2_{k}}(1))$, cf.\ \cite[\S1.5]{ABB:fibrations}.  It is a
locally free $\OO_{\P^2_{k}}$-algebra of rank 8 whose center $\CliffZ$ is
a locally free quadratic $\OO_{\P^2_{k}}$-algebra (cf.\
\cite[IV~Prop.~4.8.3]{knus:quadratic_hermitian_forms}).  The
\linedef{discriminant cover} $r : S = \SSpec \CliffZ \to \P^2$ is a
finite flat double cover branched along the sextic $D \subset \P^2_{k}$.

Assuming simple degeneration, then $S$ is a smooth K3 surface of
degree 2 over $k$.  We say that a cubic fourfold $X$ containing a
given plane $P$ is \linedef{very general} if the quadric surface
bundle associated to $P$ has simple degeneration and the associated K3
surface $S$ satisfies $r^* : \Pic(\P^2_{k}) \to \Pic(S)$ is an
isomorphism. If the ground field $k$ is algebraically closed, then for
any overfield $F$ of $k$, by rigidity of the N\'eron--Severi group,
this implies that $r^* : \Pic(\P^2_{F}) \to \Pic(S_{F})$ is an
isomorphism.  We note that in the moduli space $\mathcal{C}_8$ of
smooth cubic fourfolds containing a plane over $\C$, the locus of very
general ones (in our definition) is the complement of countably many
proper Zariski closed subvarieties (see
\cite[Thm.~1.0.1]{hassett:special-cubics}).  Thus this notion of very
general agrees with the usual notion in algebraic geometry.

A cubic fourfold $X$ containing a plane $P$ over $\C$ is very general
if and only if the Chow group $\CH^2(X)$ of cycles of codimension 2 is
spanned by the classes of $P$ and of a fiber $Q$ of the quadric
fibration (see \cite[\S1~Prop.~2]{voisin} and its proof).  A very
general cubic fourfold $X$ has nontrivial Clifford invariant.  Indeed,
the Zariski closure of a rational section of the quadric fibration
would provide a cycle of dimension 2 having one point of intersection
with a general fiber $Q$.  Such a cycle cannot be rationally
equivalent to any linear combination of $P$ and $Q$, since both $P$
and $Q$ have even intersection with $Q$. Hence $X$ cannot be very
general (see \cite[Thm.~3.1]{hassett:rational_cubic}).

Still assuming simple degeneration, the even Clifford algebra
$\CliffAlg_0$, considered over its center, defines an Azumaya
quaternion algebra $\CliffB_0$ over $S$ (cf.\
\cite[Prop.~3.13]{kuznetsov:quadrics} and \cite[Prop.~1.13]{APS:quadric_surface}).  We refer to the Brauer class
$\beta_{X,P} \in \Br(S)$ of $\CliffB_0$ as the \linedef{Clifford
invariant} of the pair $(X,P)$.  

\begin{lemma}
Let $X$ be a smooth cubic fourfold containing a plane $P$ over a field
$k$ of characteristic $\neq 2$.  Assume that the quadric surface
bundle associated to $P$ has simple degeneration. Let $S$ be the
associated K3 surface of degree 2.  If $X$ contains another plane $P'$
(necessarily skew to $P$), then the fibration $\pi : \wt{X} \to
\P^2_{k}$ has a rational section.  In this case, the generic fiber of
$\pi$ is an isotropic quadric over $k(\P^2)$, hence is a
$k(\P^2)$-rational variety.  In particular, the $k$-variety $X$ is
$k$-rational. Moreover, $\beta_{X,P}=0 \in \Br(S)$.
\end{lemma}
\begin{proof} 
Recall how the fibration $\pi: \wt{X} \to \P^2_{k}$ is constructed.  One
fixes an arbitrary plane $Q \subset \P^5_{k}$ which does not meet $P$.
The morphism $\pi$ is induced by the morphism $\varpi : X \setminus P
\to Q$ sending a point $x$ of $X$ not on $P$ to the unique point of
intersection of the linear space spanned by $P$ and $x$ with the
linear space $Q$.

Let $P'$ be another plane in $X$.  By
Proposition~\ref{prop:simpledeg}, $P'$ must be skew to $P$.  Take the
plane $Q$ to be $P'$.  On points of $P' \subset X$ the map $\varpi : X
\setminus P' \to P'$ is the identity. Thus $\pi$ has a rational
section, the generic fiber is an isotropic quadric, hence rational
over $k(\P^2)$, and the even Clifford invariant of this quadric is
trivial.
\end{proof}

In view of this lemma, when given a smooth cubic fourfold containing a
plane $P$ whose associated quadric fibration has simple degeneration,
we shall abuse notation and write $\beta_{X} \in \Br(S)$ instead of
$\beta_{X,P} \in \Br(S)$.

The generic fiber of $(\EE,q,\OO_{\P^2_{k}}(1))$ is a quadratic form
of rank 4 over $k(\P^2)$ with values in a $k(\P^2)$-vector space of
dimension 1.  Choosing a generator $l$ of $\OO_{\P^2_{k}}(1)$ over
$k(\P^2)$, we arrive at a usual quadratic form $(E,q)$ with
discriminant extension $k(S)/k(\P^2)$.  The generic fiber of $\beta_X$
is then in the image of the restriction map $\Br(k(\P^2)) \to
\Br(k(S))$ by the fundamental relations for the even Clifford algebra
(cf.\ \cite[Thm.~9.12]{KMRT}).  Explicitly, the full Clifford algebra
$C(E,q)$ is central simple over $k(\P^2)$ and its restriction to
$k(S)$ is Brauer equivalent to the even Clifford algebra $C_0(E,q)$,
hence with the generic fiber of $\beta_X$.  We note that $C(E,q)$
depends on the choice of $l$, while $C_0(E,q)$ does not.  We now
construct a particular Brauer class on $k(\P^2)$ restricting to
$\beta_X$ on $k(S)$. This will play a crucial r\^{o}le in the proof of
Theorem~\ref{thm:main}.

\begin{prop}
\label{prop:C}
Let $X$ be a smooth cubic fourfold containing a plane $P$ over a field
$k$ of characteristic $\neq 2$ .  Assume that the associated quadric
surface bundle has simple degeneration along a divisor $D \subset
\P^2_{k}$ and let $S$ be the associated K3 surface of degree 2.  Given
a choice of homogeneous coordinates on $\P^5_{k}$, there exists a line
$L \subset \P^2_{k}$ and a quaternion algebra $\alpha$ over $k(\P^2)$
whose restriction to $k(S)$ is isomorphic to the generic fiber of the
Clifford invariant $\beta_X$, and such that $\alpha$ has ramification
only at the generic points of $D$ and $L$.
\end{prop}
\begin{proof}
For a choice of homogeneous coordinates on $\P^5_{k}$, let $L \subset
\P^2_{k}$ be the line whose equation is $a_{00}$ from \eqref{eq:q}.
The smoothness of $X$ implies that $a_{00}$ is nonzero.  Then on
$\A^2_{k} = \P^2_{k} \bslash L$, the choice of $a_{00}$ determines a
trivialization $\psi : \OO(1)|_{\A^2_{k}} \to \OO_{\A^2_{k}}$, with
respect to which the quadratic form $q' = \psi \circ q|_{\A^2_{k}} :
\EE|_{\A^2_{k}} \to \OO_{\A^2_{k}}$ (given by the dehomogenization of
equation \eqref{eq:q} associated to $\psi$) represents $1$.  Letting
$V = \P^2_{k} \bslash (D \cup L) \subset \A^2_{k}$, we have that
$(\EE|_V,q'|_V,\OO_V)$ is a regular quadratic form of rank 4 on $V$.
In this case, we have the full Clifford algebra $\CliffAlg = \CliffAlg
(\EE|_V,q'|_V,\OO_V)$ at our disposal, which is an Azumaya algebra of
degree 4 on $V$.

Let us prove that $\CliffAlg|_{k(\P^2)}$ is Brauer equivalent to a
symbol.  Since $q'|_{k(\P^2)}$ represents 1, we have a diagonalization
$$
q'|_{k(\P^2)} = \quadform{1,a,b,abd}
$$
where $d \in k(\P^2)\mult/k(\P^2)\mult{}^2$ is the discriminant.  Then
in $\Br(k(\P^2))$ we have
$$
[\CliffAlg|_{k(\P^2)}] = (a,b) + (a,abd) + (b,abd) + (-1,-d) = (-ab,-ad)  
$$
by the formula \cite[Ch.~2,~Def.~12.7,Ch.~9,~Rem.~2.12]{scharlau:book}
relating the Clifford invariant to the 2nd Hasse--Witt invariant.

Letting $\alpha=(-ab,-ad) \in \Br(k(\P^2))$, we   see that
$\alpha$ coincides with $[\CliffAlg|_{k(\P^2)}]$ in $\Br(k(\P^2))$,
hence is unramified at all codimension 1 points of $V$, i.e., 
$\alpha$ is ramified  at most at the generic points of $D$ and $L$. 

The even Clifford algebra is a similarity class invariant, hence we
have an an isomorphism $\CliffAlg_0(\EE|_V,q'|_V,\OO_V) \isom
\CliffAlg_0(\EE|_V,q|_V,\OO(1)|_V) $ over $V$, hence over the inverse
image of $V$ in $S$.

 Finally,
we have $\CliffAlg|_{k(S)} \isom M_2\bigl(\CliffB_0|_{k(S)}\bigr)$,
hence $\alpha$ restricted to $k(S)$ is Brauer equivalent to the
generic fiber of the Clifford invariant $\beta_X \in \Br(S)$.
\end{proof}


\section{Proof of the main result}
\label{sec:proof}

Let us recall the statement.

\begin{theorem}
\label{thm:main_text}
Let $X \subset \P^5$ be a very general cubic fourfold containing a
plane $P \subset \P^5$ over $\C$.  Then 
$\Hur^3(X/\C,\Q/\Z(2))$ 
is universally trivial.
\end{theorem}

Let $k$ be a field of characteristic $\neq 2$.  Let $X \subset
\P^5_{k}$ be a smooth cubic fourfold containing a plane over $k$.  Let
$\pi : \wt{X} \to \P^2_{k}$ be the associated quadric surface bundle.
We assume that $\pi$ has simple degeneration along a smooth divisor $D
\subset \P^2_{k}$, see Proposition~\ref{prop:simpledeg}.  Denote by
$Q$ the generic fiber of $\pi$; it is a smooth quadric surface over
$k(\P^2)$.  For any field extension $F/k$, we will need to refer to
the following commutative diagram of Bloch--Ogus complexes
\begin{equation}
\label{eq:diagram}
\begin{split}
\xymatrix@C=8pt{
0 \ar[r] & H^3(F,\Q/\Z(2)) \ar[r] \ar[d] & H^3(F(\P^2),\Q/\Z(2))
\ar[r]^(.47){\{\partial_{\gamma}\}}  \ar@{->>}[d] & \displaystyle\bigoplus H^2(F(\gamma),\Q/\Z(1)) \ar[d]\\
0 \ar[r] & \Hur^3(X/F,\Q/\Z(2)) \ar[r] &
\Hur^3(Q_{F}/F(\P^2),\Q/\Z(2)) \ar[r]   &
\displaystyle\bigoplus H^2(F(Q_\gamma),\Q/\Z(1))
}
\end{split}
\end{equation}
where the sums are taken over all points $\gamma$ of codimension 1 of
$\P^2_F$.  The top row is the exact sequence defining the unramified
cohomology of $\P^2_{F}$ (which is constant, by
Proposition~\ref{prop:Pn}) via the residue maps $\partial_\gamma$.
The bottom row is the complex arising from taking residues on $F(Q)$
at points of codimension 1 of $\wt{X}$ whose image in $\P^2_{F}$ is of
codimension 1 on $\P^2_{F}$ (recall that
$\Hur^3(Q_{F}/F(\P^2),\Q/\Z(2))$ consists of classes over $F(Q)$ which
have trivial residues at all rank 1 discrete valuations trivial on
$F(\P^2)$).  Here, $Q_\gamma$ denotes the generic fiber of the
restricted quadric fibration $\pi|_C : \wt{X}|_C \to C$, where $C
\subset \P^2_{F}$ is a projective integral curve with generic point
$\gamma$.  Under the simple degeneration hypothesis, each $Q_\gamma$
is an integral quadric surface.  The bottom complex is also
exact. This is obvious except at the term
$\Hur^3(Q_{F}/F(\P^2),\Q/\Z(2))$, and we shall not used the exactness
at that point.  The vertical maps in the diagram are the natural ones.

Let $\psi \in \Hur^3(X_{F}/F,\Q/\Z(2)) \subset
\Hur^3(Q_{F}/F(\P^2),\Q/\Z(2))$.  By Theorem~\ref{thm:KRS}, this is
the image of an element $\xi \in H^3(F(\P^2),\Q/\Z(2))$.  By the
commutativity of the right-hand square, and since the bottom row of
diagram \eqref{eq:diagram} is a complex, for each $\gamma$ of
codimension 1 in $\P^2_{F}$ we know that
\begin{equation}
\label{eq:inker}
\partial_\gamma(\xi) \in \ker\bigl(H^2(F(\gamma),\Q/\Z(1)) \to
H^2(F(Q_\gamma),\Q/\Z(1))\bigr).
\end{equation}
We recall that if $F$ is any field, then $H^2(F,\Q/\Z(1))$ is
isomorphic to the prime-to-$p$ part of the Brauer group $\Br(F)$,
where $p$ is the characteristic of $F$.  With the notation of \S
\ref{sec:Clifford_algebras}, where the line $L$ is defined by
$a_{00}=0$, let $d = \Delta/a_{00}^6$. Then $\divisor(d) = D - 6L$,
and the class of $d$ in $k(\P^2)/k(\P^2)^{\times 2}$ is the
discriminant of the quadric $Q$.

\begin{defin} 
Fix $L$ as above.  Let $F/k$ be a field extension and $\xi$ a class in
$H^3(F(\P^2),\Q/\Z(2))$.  We call an integral curve $C \subset \P^2_F$
with generic point $\gamma$ a \linedef{bad curve} (for $\xi$) if $C$
is different from $D_F$ and $L_F$ and if $\partial_\gamma(\xi) \neq 0$
in $H^2(F(\gamma),\Q/\Z(1))$.
\end{defin}

There are finitely many bad curves for each given $\xi \in
H^3(F(\P^2),\Q/\Z(2))$.  Let $\alpha \in H^2(k(\P^2),\mu_2)$ be the
class of a quaternion algebra attached to $\pi : \wt{X} \to \P^2_{k}$
and the choice of a line $L$, as in Proposition~\ref{prop:C}.
Theorem~\ref{thm:kernel} and Proposition~\ref{prop:kernel_cone} imply
that the following statements hold concerning bad curves:
\begin{enumerate}
\item The class $d|_\gamma \in H^1(F(\gamma),\mu_2)$ is trivial.

\item The class $\alpha|_\gamma \in H^2(F(\gamma),\mu_2)$ is
nontrivial and coincides with $c(Q)|_\gamma \in \Br(F(\gamma))$.

\item The class $\partial_\gamma(\xi) \in H^2(F(\gamma),\Q/\Z(1))$
also coincides with $c(Q)|_\gamma \in \Br(F(\gamma))$.
\end{enumerate}

For curves $C$ split by the discriminant extension (e.g., for bad
curves), we will construct special rational functions that are
parameters along $C$ and are norms from the discriminant extension.
This is where the very general hypothesis will be used.

\begin{lemma}
\label{lem:bad6}
Let $k$ be a field of characteristic $\neq 2$.  Let $r : S \to
\P^2_{k}$ be a finite flat morphism of degree 2 branched over a smooth
sextic curve $D$.  Assume that $r^* : \Pic(\P^2_{k}) \to \Pic(S)$ is
an isomorphism.  Choose a line $L \subset \P^2_{k}$ and a function $d
\in k(\P^2)$ with divisor $D-6L$ that satisfies
$k(S)=k(\P^2)(\sqrt{d})$.  If $C \subset \P^2_{k}$ is an integral
curve (with generic point $\gamma$) different from $D$ and $L$ such
that $d|_\gamma \in H^1(k(\gamma),\mu_2)$ is trivial, then there
exists a function $f_\gamma \in k(\P^2)$ whose divisor is $C-2nL$ for
some positive integer $n$, and which is a
a norm from $k(S)/k(\P^2)$.
\end{lemma}
\begin{proof}
  Under our hypothesis, $\Pic(S) =
\CH^1(S) = \Z H$, where $H=r^* L$ by flat pull-back.  As a
consequence, the action of $\Aut(S/\P^2)$ on $\Pic(S)=\CH^1(S)$ is
trivial.

If the class of $d$ is a square in $k(\gamma)$, then $f\inv C$ splits
as $C_{1} \cup C_{2}$, with $C_{1}$ and $C_{2}$ both having class $nH$
for some positive integer $n$.  Hence $C_1 - nH = \divisor(g_\gamma)$
for some $g_\gamma \in k(S)$.  However, by proper push-forward,
$r_{*}(C_1-nH) = C-2nL$, hence $C-2nL=
\divisor(N_{k(S)/k(\P^2)}(g_{\gamma}))$.  Set $f_\gamma =
N_{k(S)/k(\P^2)}(g_{\gamma})$.
\end{proof}

\begin{remark}
From the proof of Lemma~\ref{lem:bad6}, one sees that every bad curve
$C$ has even degree.  Using some further intersection theory, one can
even prove that any such curve has degree 6, though we shall not need
this.
\end{remark}

Now assume that $k=\C$.  Let $\alpha \in \Br(\C(\P^2))$ be the class
of the full Clifford algebra $(\EE,q,\OO(1))$ over $\P^2 \bslash (D
\cup L)$, as in Proposition~\ref{prop:C}.  Then $\alpha$ is
unramified, except possibly at the generic points of $D$ and $L$.

Now assume that $X$ is very general.  Fix a field $F/\C$.  For each
bad curve $C$ over $F$, we will choose a parameter $f_\gamma \in
F(\P^2)$ at the generic point $\gamma$ of $C$ as in
Lemma~\ref{lem:bad6}, which we may apply since $r^* : \Pic(\P^2_{F})
\to \Pic(S_{F})$ is an isomorphism by the rigidity of the
N\'eron--Severi group.  Denote by
$$
f = \prod f_\gamma,
$$
where the product is taken over all bad curves.


Let us continue with the proof of Theorem~\ref{thm:main}.
Recall that we have  lifted 
$\psi \in \Hur^3(X_{F}/F,\Q/\Z(2))$ 
to an
element $\xi \in H^3(F(\P^2),\Q/\Z(2))$.
We now compute the residues of $\xi$ along codimension one points of
$\P^2_F$, where we need only worry about the generic points of $D$,
any bad curves $C$, and $L$.

Let us first consider the generic point $\eta$ of $D$.  The hypothesis
that $X$ is very general implies simple degeneration, hence that the
$F(\eta)$-quadric $Q_\eta$ is a quadric cone over a smooth
$F(\eta)$-conic.  The conic is split by Tsen's theorem since it is the
base change of a smooth $\C(\eta)$-conic.  Hence
Proposition~\ref{prop:kernel_cone}, and the commutativity of the right
hand square of diagram \eqref{eq:diagram}, implies that $\xi$ is
unramified at $\eta$.

Now, we will compare $\xi$ with the class $\alpha \cup (f)$, whose
ramification we control.  
Relying on Theorem \ref{MSBK},  
we are implicitly considering an inclusion
$H^3(F(\P^2),\mu_2^{\tensor 2}) \subset H^3(F(\P^2),\Q/\Z(2))$.
By construction, the function $f$ is a norm from the extension
$F(\P^2)(\sqrt{d})/F(\P^2)$, i.e., is of the form $f = g^2-dh^2$ for
some $g, h \in F(\P^2)$. Also, $f$ has its zeros and poles only along
the bad curves and $L$, hence in particular, $f$ is a unit at the
generic point $\eta$ of $D$.

The extension $F(\P^2)(\sqrt{d})/F(\P^2)$ is
totally ramified at $\eta$.  In particular, any unit in the local ring
at $\eta$ which is a norm from $F(\P^2)(\sqrt{d})$ reduces to a square
in the residue field $F(\eta)$.  Thus $f$ lifts to a square in the
completion $\widehat{F(\P^2})_{\eta}$.

Now we consider the residues of $\alpha \cup (f) \in
H^3(F(\P^2),\mu_2^{\tensor 2})$ along codimension one points of
$\P^2_F$.  Both $\alpha \in H^2(F(\P^2),\mu_2)$ and $(f)
\in H^1(F(\P^2),\mu_2)$ are unramified away from the generic points of
$D$, $L$, and the bad curves $C$.  

At the generic point $\eta$ of $D$, the function $f$ is a square in
the completion $\widehat{F(\P^2})_{\eta}$.  Thus, the residue of
$\alpha \cup (f)$ at $\eta$ is zero.  At a bad curve, $\alpha$ is
regular and the valuation of $f$ is one. Thus the residue at such a
curve is $\alpha|_\gamma = c(Q)|_\gamma$.

Thus, we have that the difference $\xi - \alpha \cup (f) \in H^3(F(\P^2),\Q/\Z(2))$ has trivial residues away
from $L$, hence it comes from a constant class $\xi_0$ in the image of
$H^3(F,\Q/\Z(2)) \to H^3(F(\P^2),\Q/\Z(2))$.

Now we show that $\alpha \cup (f)$ vanishes when restricted to
$H^3(F(Q),\mu_2^{\tensor 2})$.  In the notation of the proof of
Proposition~\ref{prop:C}, $q|_{F(\P^2)} = \quadform{1,a,b,abd}$
becomes isotropic over $F(Q)$ and $\alpha=(-ab,-ad)$.  Since $f$ is a
norm from $F(\P^2)(\sqrt{d})$, Lemma~\ref{lem:ct} implies that
$(a,b,f)$ is trivial in $H^3(F(Q),\mu_2^{\tensor 2})$.  Further, since
$f$ is a norm from $F(\P^2)(\sqrt{d})$, we have $(d,f)=0$.

 But then
$$
\alpha \cup (f) = (ab,ad,f) = (a,b,f) + (a,a,f) + (ab,d,f)
$$
is trivial in $H^3(F(Q),\mu_2^{\tensor 2})$ as well.  Thus 
$\psi \in \Hur^3(X_{F}/F,\Q/\Z(2))$
is the image of the constant class $\xi_0 \in H^3(F,\Q/\Z(2))$.


\begin{remark}
We can give a different argument, using Arason's
result (see Theorem~\ref{thm:arason}), for the vanishing of $\alpha \cup
(f)$ in $H^3(F(Q),\mu_2^{\tensor 2})$.  As in the notation of the
proof of Proposition~\ref{prop:C}, we write $q_{k(\P^2)} =
\;\quadform{1,a,b,abd}$ and $\alpha = (-ab,-ad)$.  The 3-Pfister form
associated to $\alpha \cup (f)$ decomposes as
$$ 
\Pfister{-ab,-ad,f} \;=\; \quadform{1,ab,ad,bd} \perp -f
\!\quadform{1,ab,ad,bd} 
$$
Since $f$ is a norm from $F(\P^2)(\sqrt{d})$ we have that $d$ is a
norm from $F(\P^2)(\sqrt{f})$.  Thus $d$ is a similarity factor of the
norm form of $F(\P^2)(\sqrt{f})/F(\P^2)$, i.e., we have an isometry
$\quadform{1,-f}\, \isom \,\quadform{d,-df}$. Hence the 3-Pfister form
\begin{align*}
\Pfister{-ab,-ad,f} & =\; \quadform{1,ab, ad, bd}  \perp  -f \quadform{1,ab,ad,bd} \\ 
 & = \; \quadform{ab,ad,bd} \perp \quadform{1,-f} \perp -f\quadform{ab,bd,ad} \\
 & = \; \quadform{ab,ad,bd} \perp \quadform{d,-df} \perp -f\quadform{ab,bd,ad} \\
 & = \; \quadform{d,ab,ad,bd} \perp -f\quadform{d,ab,bd,ad}
\end{align*}
contains the form $\quadform{d,ad,bd,ab} \; =
d\!\quadform{1,a,b,abd}$.  Thus by Theorem~\ref{thm:KRS}, we have that
$\alpha \cup (f)$ is trivial in
$\Hur^3(Q_{F}/F(\P^2),\Q/\Z(2))$. \end{remark}

\begin{cor}
\label{cor:mainZ/2}
Let $X \subset \P^5$ be a very general cubic fourfold containing a
plane $P \subset \P^5$ over $\C$.  Then $\Hur^3(X/\C,\mu_2^{\tensor
2})$ is universally trivial.
\end{cor}
\begin{proof}
In the following commutative diagram
$$
\xymatrix @R=14pt{
H^3(F,\mu_{2}^{\otimes 2})  \ar[d] \ar[r]  &  \Hur^3(X_{F}/F,\mu_{2}^{\otimes 2}) \ar[d] \\
H^3(F,\Q/\Z(2))  \ar[r]   & \Hur^3(X_{F}/F,\Q/\Z(2)  ) 
}
$$
the horizontal maps are injective since we may specialize to an
$F$-rational point. The vertical maps are injective by Theorem
\ref{MSBK}.  By Theorem~\ref{thm:main_text}, any $\psi \in
\Hur^3(X_{F}/F,\mu_2^{\tensor 2}) \subset \Hur^3(X_{F}/F,\Q/\Z(2))$ is
the image of a constant class $\xi_0 \in H^3(F,\Q/\Z(2))$, which by
the diagram is 2-torsion, hence comes from an element in
$H^3(F,\mu_{2}^{\otimes 2})$.  Since the right hand side vertical map
is injective, the proof is complete.
\end{proof}


\providecommand{\bysame}{\leavevmode\hbox to3em{\hrulefill}\thinspace}
\providecommand{\href}[2]{#2}

\end{document}